\documentclass[12pt, a4paper]{amsart}
\usepackage[utf8]{inputenc}
\usepackage{amsxtra,amssymb,amsthm,amsmath,amscd,mathrsfs, epsfig, eufrak}
\usepackage{amscd, amsmath, mathrsfs, amssymb, amsthm, amsxtra, bbding, epsfig, eucal, eufrak, graphicx, latexsym, mathrsfs, mathbbol, bbold}
\usepackage[all]{xy}
\usepackage{tikz}

\usepackage[normalem]{ulem}
\usepackage{soul}
\usepackage{color}

\setstcolor{red}

\makeatletter
\@namedef{subjclassname@2010}{%
  \textup{2010} Mathematics Subject Classification}
\makeatother

\def \N {{\mathbb N}}

\def \d {\,{\rm d}}

\def\Li{\hbox{{\rm Li}}}

\def\le{\leqslant}
\def\ge{\geqslant}

\theoremstyle{plain}
\newtheorem{theorem}{Theorem}
\newtheorem{proposition}{Proposition}[section]
\newtheorem{lemma}[proposition]{Lemma}

\theoremstyle{remark}

\numberwithin{equation}{section}

\addtocounter{footnote}{1}

\begin{document}

\vglue -5mm

\title[\tiny Distribution of elements of a floor function set in arithmetical progression]
{Distribution of elements of a floor function set in arithmetical progression}
\author{\tiny Yahui Yu and Jie Wu}

\address{%
Yahui Yu
\\
Department of Mathematics and Physics
\\
Luoyang Institute of Science and Technology
\\
Luoyang, Henan 471023
\\
P. R. China
}
\email{yuyahui@lit.edu.cn}

\address{%
Jie Wu
\\
CNRS UMR 8050
\\
Laboratoire d'Analyse et de Math\'ematiques Appliqu\'ees
\\
Universit\'e Paris-Est Cr\'eteil
\\
94010 Cr\'eteil cedex
\\
France}
\email{jie.wu@u-pec.fr}

\date{\today}

\begin{abstract}
Let $[t]$ be the integral part of the real number $t$.
The aim of this short note is to study the distribution of elements of the set 
$\mathcal{S}(x) := \{[\frac{x}{n}] : 1\le n\le x\}$ 
in the arithmetical progression $\{a+dq\}_{d\ge 0}$.
Our result is as follows: the asymptotic formula
\begin{equation}\label{YW:result}
S(x; q, a)
:= \sum_{\substack{m\in \mathcal{S}(x)\\ m\equiv a ({\rm mod}\,q)}} 1 
= \frac{2\sqrt{x}}{q} + O((x/q)^{1/3}\log x)
\end{equation}
holds uniformly for $x\ge 3$, $1\le q\le x^{1/4}/(\log x)^{3/2}$ and $1\le a\le q$,
where the implied constant is absolute.
The special case of \eqref{YW:result} with fixed $q$ and $a=q$ confirms a recent numeric test of Heyman.

\end{abstract}


\maketitle   	

\section{Introduction}

As usual, denote by $[t]$ the integral part of the real number $t$.
Recently Heyman \cite{Heyman2019} quantifies the cardinal of the set
\begin{equation}\label{def:S(x)}
\mathcal{S}(x) 
:= \Big\{\Big[\frac{x}{n}\Big] : 1\le n\le x\Big\}.
\end{equation}
Noticing that
$$
\mathcal{S}(x) 
= \Big\{m \in \N : m=\Big[\frac{x}{n}\Big] \; \text{for some} \; n\le x\Big\},
$$
by elementary argument he proved that (see \cite[Theorems 1 and 2]{Heyman2019}):
\begin{equation}\label{Heyman:result-1}
S(x) := |\mathcal{S}(x)|
= [\sqrt{4x+1}]-1
= 2\sqrt{x} + O(1) 
\end{equation}
for $x\to\infty$.
Subsequently in his another article \cite[Theorem 1]{Heyman2021}, 
he also investigated the number of primes in the set $\mathcal{S}(x)$:
\begin{equation}\label{Heyman:result-2}
\begin{aligned}
\pi_{\mathcal{S}}(x) 
& := \Big|\Big\{\Big[\frac{x}{n}\Big] : 1\le n\le x \;\,\text{and}\; \Big[\frac{x}{n}\Big] \; \text{is prime}\Big\}\Big|
\\\noalign{\vskip 0,5mm}
& \; = \frac{2\sqrt{x}}{\log\sqrt{x}} + O\bigg(\frac{\sqrt{x}}{(\log x)^2}\bigg),
\qquad\text{as}\quad
 x\to\infty.
\end{aligned}
\end{equation}
This can be considered as analogue of the prime number theorem for the set $\mathcal{S}(x)$.
Very recently Ma and Wu \cite{MaWu2021} sharpened this result by proving the strong form of the prime number theorem for $\mathcal{S}(x)$: 
\begin{equation}\label{MaWu:result-1}
\pi_{\mathcal{S}}(x) 
= \Li_{\mathcal{S}}(x) + O\big(\sqrt{x}\,{\rm e}^{-c(\log x)^{3/5}(\log\log x)^{-1/5}}\big),
\end{equation}
where $c>0$ is a positive constant and
$$
\Li_{\mathcal{S}}(x) := \int_2^{\sqrt{x}} \frac{\d t}{\log t} + \int_2^{\sqrt{x}} \frac{\d t}{\log(x/t)}\cdot
$$
Some related results have been obtained in \cite{LiuWuYang2021a, MaWu2020}.

In \cite{Heyman2019}, Heyman also proposed to consider a more general problem than \eqref{Heyman:result-1}, 
i.e. to study the asymptotic behaviour of the cardinal $S(x; q)$ of the set
$$
\mathcal{S}(x; q)
:= \Big\{\Big[\frac{x}{n}\Big] : 1\le n\le x \;\, \text{and}\;\, q\mid \Big[\frac{x}{n}\Big]\Big\}
$$
for each fixed integer $q\ge 1$.
Let $\psi(t) := t-[t]-\frac{1}{2}$ and $b:=(\sqrt{4x+1}-1)/2$.
Heyman first showed that (see \cite[Lemma 3]{Heyman2019})
$$
S(x; q) 
= \frac{4\sqrt{x}}{3q}
+ \sum_{r=1}^{[x/b]} \sum_{d=(x-r)/(qr)}^{x/(qr)} \Big(\psi\Big(\frac{x}{dq+1}\Big)-\psi\Big(\frac{x}{dq}\Big)\Big)
+ O(1)
$$
and then wrote:
``Calculating
various sums using Maple suggests that the double sum cannot successfully be
bound. In fact Maple suggests that the double sum is asymptotically equivalent to
$2\sqrt{x}/(3q)$. If this argument is correct then
\begin{equation}\label{Conjecture:Heyamn}
S(x; q)\sim \frac{2\sqrt{x}}{q}
\qquad
(x\to\infty),
\end{equation}
as one would expect heuristically."

The aim of this short note is to prove \eqref{Conjecture:Heyamn}.
In fact we can consider a more general problem: for $1\le a\le q$,
study the distribution of elements of the set $\mathcal{S}(x)$ 
in the arithmetical progression $\{a+dq\}_{d\ge 0}$.
Define
$$
\mathcal{S}(x; q, a)
:= \Big\{\Big[\frac{x}{n}\Big] : 1\le n\le x \;\, \text{and}\;\, \Big[\frac{x}{n}\Big]\equiv a\,({\rm mod}\,q)\Big\}
$$
and
$$
S(x; q, a) := |\mathcal{S}(x; q, a)|.
$$

Our result is as follows.

\begin{theorem}\label{thm1}
Under the previous notation, we have
\begin{equation}\label{thm1:eq-1}
S(x; q, a) = \frac{2\sqrt{x}}{q} + O((x/q)^{1/3}\log x)
\end{equation}
uniformly for $x\ge 3$, $1\le q\le x^{1/4}/(\log x)^{3/2}$ and $1\le a\le q$,
where the implied constant is absolute.
\end{theorem}

Since $S(x; q, q) = S(x; q)$, Heyman's expected result \eqref{Conjecture:Heyamn} is a special case of 
our \eqref{thm1:eq-1} with fixed $q$.
It is worth to notice that le result of Theorem \ref{thm1} implies  
$$
S(x; q, a)\sim \frac{2\sqrt{x}}{q}
$$ 
uniformly for $1\le q = o(x^{1/4}/(\log x)^{3/2})$ and $1\le a\le q$.
We didn't make an effort to get the best possible exponent,
and further improvements of the constant $\frac{1}{4}$ are possible.

It seems interesting to establish analogues of the Dirichlet theorem or more general the Siegel-Walfisz theorem,
the Brun-Titchmarsh theorem, the Bombieri-Vinogradov theorem for the set $\mathcal{S}(x)$. 
We shall  leave these problems to another occasion.

\vskip 8mm

\section{Preliminary lemmas}

In this section, we shall cite two lemmas, which will be needed in the next section.
The first one is due to Vaaler (see \cite[Theorem A.6]{GrahamKolesnik1991}).

\begin{lemma}\label{lem:Vaaler}
Let $\psi(t) := t - [t]-\frac{1}{2}$.
For $x\ge 1$ and $H\ge 1$, we have
$$
\psi(x) = - \sum_{1\le |h|\le H} \Phi\Big(\frac{h}{H+1}\Big) \frac{{\rm e}(hx)}{2\pi{\rm i} h} + R_H(x),
$$
where ${\rm e}(t):={\rm e}^{2\pi{\rm i}t}$, $\Phi(t):=\pi t (1-|t|)\cot(\pi t) + |t|$ and the error term $R_H(x)$ satisfies
\begin{equation}\label{lem:Vaaler:eq}
|R_H(x)|\le \frac{1}{2H+2} \sum_{0\le |h|\le H} \Big(1-\frac{|h|}{H+1}\Big) {\rm e}(hx).
\end{equation}
\end{lemma}

\begin{lemma}\label{lem2}
For $1\le a\le q$ and $\delta\in \{0, 1\}$, define
\begin{equation}\label{lem2:eq1}
\mathfrak{S}_{\delta}(D, D')
:= \sum_{D<d\le D'} \psi\Big(\frac{x}{dq+a+\delta}\Big).
\end{equation}
If $(\kappa, \lambda)$ is an exponent pair, then we have
\begin{equation}\label{lem2:eq2}
\mathfrak{S}_{\delta}(D, D')
\ll (x^{\kappa} D^{-\kappa+\lambda} q^{-\kappa})^{1/(1+\kappa)} + x^{\kappa} D^{-2\kappa+\lambda} q^{-\kappa} + x^{-1} D^2 q
\end{equation}
uniformly for $1\le a\le q$, $(\sqrt{x}-a)/q<D\le (x/q)^{2/3}$ and $D<D'\le 2D$,
where the implied constant depends on $(\kappa, \lambda)$ at most.
\end{lemma}

\begin{proof}
Using Lemma \ref{lem:Vaaler}, we can write
\begin{equation}\label{proof:prop_2:1}
\mathfrak{S}_{\delta}(D, D')
= - \frac{1}{2\pi\text{i}} \big(\mathfrak{S}_{\delta}^{\flat}(D, D')
+ \overline{\mathfrak{S}_{\delta}^{\flat}(D, D')}\big)
+ \mathfrak{S}_{\delta}^{\dagger}(D, D'),
\end{equation}
where $H\le D$ and 
\begin{align*}
\mathfrak{S}_{\delta}^{\flat}(D, D')
& := \sum_{h\le H} \frac{1}{h}\Phi\big(\frac{h}{H+1}\big)
\sum_{D<d\le D'} \text{e}\Big(\frac{hx}{dq+a+\delta}\Big),
\\
\mathfrak{S}_{\delta}^{\dagger}(D, D')
& := \sum_{D<d\le D'} R_H\Big(\frac{x}{dq+a+\delta}\Big).
\end{align*}
Inverting the order of summations and applying the exponent pair $(\kappa, \lambda)$ to the sum over $d$,
it follows that
\begin{equation}\label{proof:prop_2:2}
\begin{aligned}
\mathfrak{S}_{\delta}^{\flat}(D, D')
& \ll \sum_{h\le H} \frac{1}{h} 
\Big(\Big(\frac{xh}{D^2q}\Big)^{\kappa} D^{\lambda} + \Big(\frac{xh}{D^2q}\Big)^{-1}\Big)
\\\noalign{\vskip 1mm}
& \ll x^{\kappa} D^{-2\kappa+\lambda} q^{-\kappa} H^{\kappa} + x^{-1} D^2q.
\end{aligned}
\end{equation}
On the other hand, \eqref{lem:Vaaler:eq} of Lemma \ref{lem:Vaaler} allows us to derive that
\begin{align*}
\big|\mathfrak{S}_{\delta}^{\dagger}(D, D')\big|
& \le \sum_{D<d\le D'} \Big|R_H\Big(\frac{x}{dq+a+\delta}\Big)\Big|
\\
& \le \frac{1}{2H+2} \sum_{0\le |h|\le H} \Big(1-\frac{|h|}{H+1}\Big) 
\sum_{D<d\le D'} {\rm e}\Big(\frac{xh}{dq+a+\delta}\Big).
\end{align*}
When $h\not=0$, as before we apply the exponent pair $(\kappa, \lambda)$ to the sum over $d$ and obtain
\begin{equation}\label{proof:prop_2:3}
\mathfrak{S}_{\delta}^{\dagger}(D, D')
\ll DH^{-1} + x^{\kappa} D^{-2\kappa+\lambda} q^{-\kappa} H^{\kappa} + x^{-1} D^2q.
\end{equation}
Inserting \eqref{proof:prop_2:2} and \eqref{proof:prop_2:3} into \eqref{proof:prop_2:1}, it follows that
$$
\mathfrak{S}_{\delta}(D, D')
\ll DH^{-1} + x^{\kappa} D^{-2\kappa+\lambda} q^{-\kappa} H^{\kappa} + x^{-1} D^2q
$$
for $H\le D$. Optimising $H$ on $[1, D]$, we obtain the required inequality \eqref{lem2:eq2}.
\end{proof}

\vskip 8mm

\section{Proof of Theorem \ref{thm1}}

If $\big[\frac{x}{n}\big] = m = dq+a$ with $0\le d\le (x-a)/q$, then $x/(dq+a+1)<n\le x/(dq+a)$. 
Thus we can write
\begin{equation}\label{M:expression}
\begin{aligned}
S(x; q, a)
& = \sum_{d\le (x-a)/q} \mathbb{1}\Big(\Big[\frac{x}{dq+a}\Big]-\Big[\frac{x}{dq+a+1}\Big]>0\Big) + O(1)
\\
& = S_1(x; q, a) + S_2(x; q, a) + O(1),
\end{aligned} 
\end{equation} 
where $\mathbb{1}=1$ if the statement is true and 0 otherwise, and
\begin{align*}
S_1(x; q, a)
& := \sum_{d\le (\sqrt{x}-a)/q} \mathbb{1}\Big(\Big[\frac{x}{dq+a}\Big]-\Big[\frac{x}{dq+a+1}\Big]>0\Big),
\\
S_2(x; q, a)
& := \sum_{(\sqrt{x}-a)/q<d\le (x-a)/q} \mathbb{1}\Big(\Big[\frac{x}{dq+a}\Big]-\Big[\frac{x}{dq+a+1}\Big]>0\Big).
\end{align*}

For $d\le (\sqrt{x}-a-1)/q$, we have
$$
\Big[\frac{x}{dq+a}\Big]-\Big[\frac{x}{dq+a+1}\Big]
>\frac{x}{(dq+a)(dq+a+1)}-1
>0.
$$
Thus we have
\begin{equation}\label{Mdagger:result}
S_1(x; q, a) = \frac{\sqrt{x}}{q} + O(1)
\end{equation}
for $x\ge 3$, where the implied constant is absolute.

Next we treat $S_2(x; q, a)$.
Noticing that for $d>(\sqrt{x}-a)/q$ we have
$$
0<\frac{x}{dq+a}-\frac{x}{dq+a+1}=\frac{x}{(dq+a)(dq+a+1)}<1,
$$
the quantity $\big[\frac{x}{dq+a}\big]-\big[\frac{x}{dq+a+1}\big]$ can only equal to 0 or 1. 
On the other hand, for $d\ge (x/q)^{2/3}$, then $dq+a=[\frac{x}{n}]$ for some $n\le (x/q)^{1/3}$.
Thus we can write
$$
S_2(x; q, a)
= \sum_{(\sqrt{x}-a)/q<d\le (x/q)^{2/3}} \Big(\Big[\frac{x}{dq+a}\Big]-\Big[\frac{x}{dq+a+1}\Big]\Big) + O((x/q)^{1/3}).
$$
Noticing that
$$
\Big[\frac{x}{dq+a}\Big]-\Big[\frac{x}{dq+a+1}\Big]
= \frac{x}{dq+a}-\frac{x}{dq+a+1}-\psi\Big(\frac{x}{dq+a}\Big)+\psi\Big(\frac{x}{dq+a+1}\Big),
$$
it follows that
\begin{equation}\label{Msharp:decomposition}
S_2(x; q, a)
= S_{2, 1}(x; q, a) - S_{2, 2}^{\langle 0\rangle}(x; q, a) + S_{2, 2}^{\langle 1\rangle}(x; q, a) 
+ O((x/q)^{1/3}),
\end{equation}
where 
\begin{align*}
S_{2, 1}(x; q, a)
& := \sum_{(\sqrt{x}-a)/q<p\le (x/q)^{2/3}} \Big(\frac{x}{dq+a}-\frac{x}{dq+a+1}\Big),
\\
S_{2, 2}^{\langle \delta\rangle}(x; q, a)
& := \sum_{(\sqrt{x}-a)/q<d\le (x/q)^{2/3}} \psi\Big(\frac{x}{dq+a+\delta}\Big).
\end{align*}
Firstly an elementary computation shows that
\begin{equation}\label{M1-sharp}
S_{2, 1}(x; q, a)
= \sum_{(\sqrt{x}-a)/q<d\le (x/q)^{2/3}} \frac{x}{d^2q^2} + O(1)
= \frac{\sqrt{x}}{q} + O\big((x/q^4)^{1/3}\big)
\end{equation}
for $x\to\infty$.

It remains to bound $S_{2, 2}^{\langle \delta\rangle}(x; q, a)$.
According to \cite[Theorem 3.10]{GrahamKolesnik1991}),
$(\frac{1}{2}, \frac{1}{2})$ is an exponent pair.
Thus we can take $(\kappa, \lambda) = (\frac{1}{2}, \frac{1}{2})$ in \eqref{lem2:eq2} of Lemma \ref{lem2} to get
$$
\mathfrak{S}_{\delta}(D, D')
\ll (x/q)^{1/3} + (xD^{-1}q^{-1})^{1/2} + x^{-1} D^2 q
$$
uniformly for $1\le a\le q$, $(\sqrt{x}-a)/q<D\le (x/q)^{2/3}$ and $D<D'\le 2D$.
Using $(\sqrt{x}-a)/q<D\le (x/q)^{2/3}$, we easily see that the preceding inequality implies
$$
\mathfrak{S}_{\delta}(D, D')
\ll (x/q)^{1/3}
$$
uniformly for $1\le a\le q$, $(\sqrt{x}-a)/q<D\le (x/q)^{2/3}$ and $D<D'\le 2D$.
From this, we can derive that
\begin{equation}\label{M2-sharp,delta}
\begin{aligned}
S_{2, 2}^{\langle \delta\rangle}(x; q)
& \ll (\log x) \max_{(\sqrt{x}-a)/q<D\le (x/q)^{2/3}} |\mathfrak{S}_{\delta}(D, 2D)|
\\
& \ll (x/q)^{1/3}\log x.
\end{aligned}
\end{equation}
Inserting \eqref{M1-sharp} and \eqref{M2-sharp,delta} into \eqref{Msharp:decomposition}, we obtain
\begin{equation}\label{Msharp:result}
S_2(x; q, a)
= \frac{\sqrt{x}}{q} + O((x/q)^{1/3}\log x)
\end{equation}
for $x\to\infty$.

Now the required result follows from \eqref{M:expression}, \eqref{Mdagger:result} and \eqref{Msharp:result}.
\hfill
$\square$

\vskip 8mm

\noindent{\bf Acknowledgement}. 
We begot to work this project when the second author visited
Luoyang Institute of Science and Technology in summer of 2021, 
despite the difficulty caused by the Coronavirus. 
It is a pleasure to record his gratitude to this institution for their hospitality and support. 
This work is in part supported by the National Natural Science Foundation of China 
(Grant Nos. 11771211, 11971370 and 12071375),
by the Young talent-training plan for college teachers in Henan province (2019GGJS241).


\vskip 8mm

\begin{thebibliography}{150}

\bibitem{GrahamKolesnik1991}
S. W. Graham and G. Kolesnik,
\emph{Van der Corput's Method of Exponential Sums},
Cambridge University Press,
Cambridge, 1991.

\bibitem{Heyman2019}
R. Heyman, 
\emph{Cardinality of a floor function set},
Integers, {\bf 19} (2019), A67.
See also: {\tt arXiv:1905.00533v2 [math.NT] 16 May 2019.}

\bibitem{Heyman2021}
R. Heyman,
\emph{Primes in floor function sets},
{\tt arXiv:2111.00408v4 [math.NT] 2 Dec 2021.}

\bibitem{LiuWuYang2021a}
K. Liu, J. Wu \& Z.-S. Yang,    
\emph{A variant of the prime number theorem},
Indag. Math. to appear,
{\tt https://doi.org/10.1016/j.indag.2021.09.005}

\bibitem{MaWu2020}
J. Ma \& J. Wu,    
\textit{On a sum involving the von Mangoldt function},
Period. Math. Hung. {\bf 83} (2021), Number 1, 39--48.

\bibitem{MaWu2021}
R. Ma \& J. Wu,
\emph{On the primes in floor function sets},
{\tt arXiv:2112.12426 [math.NT] 23 Dec 2021.}

\end{thebibliography}
\end{document}